\documentclass[11pt]{amsart}

\usepackage{fullpage}

 \usepackage{amsfonts}
\usepackage{amsmath}
\usepackage{amsthm}
\usepackage{mathrsfs}
\usepackage[v2]{xy}
\usepackage{amscd}  
\usepackage{amssymb}
\usepackage{graphicx}
\usepackage{enumerate}
\usepackage{color}
\usepackage{setspace}

\newtheorem{theorem}{Theorem}[section]
\newtheorem{corollary}[theorem]{Corollary} 
\newtheorem{lemma}[theorem]{Lemma}
\newtheorem{proposition}[theorem]{Proposition} 

\theoremstyle{definition}
\newtheorem{definition}[theorem]{Definition}

\theoremstyle{example}

\theoremstyle{remark}
\newtheorem{remark}[theorem]{Remark}

\numberwithin{equation}{section}

\newcommand{\hs}{\mathcal{H}}
\newcommand{\dd}{\,\textrm{d}}

\def\C{{\mathbb C}}
\def\R{{\mathbb R}}
 
\def\Z{{\mathbb Z}}
\def\N{{\mathbb N}}
\def\T{{\mathbb T}}

\begin{document}
\allowdisplaybreaks[4]

\title[Scaled convolution measure]{Scaling by $5$ on a $\frac{1}{4}$--Cantor measure}
\author[P.E.T. Jorgensen]{Palle E. T. Jorgensen}
\address[Palle E.T. Jorgensen]{Department of Mathematics, The University of Iowa, Iowa
City, IA 52242-1419, U.S.A.}
\email{jorgen@math.uiowa.edu}
\urladdr{http://www.math.uiowa.edu/\symbol{126}jorgen/}
\author[K.A. Kornelson]{Keri A. Kornelson}
\address[Keri Kornelson]{Department of Mathematics, The University of Oklahoma, Norman, OK, 73019-0315, U.S.A.}
\email{kkornelson@math.ou.edu}
\urladdr{http://www.math.ou.edu/\symbol{126}kkornelson/}
\author[K.L. Shuman]{Karen L. Shuman}
\address[Karen Shuman]{Department of Mathematics and Statistics,
Grinnell College, Grinnell, IA 50112-1690, U.S.A.}
\email{shumank@math.grinnell.edu}
\urladdr{http://www.math.grinnell.edu/\symbol{126}shumank/}
\thanks{The second and third authors were supported in part by NSF grant DMS-0701164.  The third author was supported in part by the Grinnell College Committee for the Support of Faculty Scholarship.}

\dedicatory{To the memory of William B. Arveson}

\subjclass[2000]{26A30; 42A63; 42A85, 46L45, 47L60, 58C40}

\keywords{Bernoulli convolution;  Cantor measure; Hilbert space; unbounded operators, fractal measures, Fourier expansions, spectral theory, decomposition theory}

\begin{abstract} 
Each Cantor measure $\mu$ with scaling factor $\frac{1}{2n}$ has at least one associated orthonormal basis of exponential functions (ONB) for $L^2(\mu)$.   In the particular case where the scaling constant for the Cantor measure is $\frac14$ and two specific ONBs are selected for $L^2(\mu_{\frac14})$, there is a unitary operator $U$ defined by mapping one ONB to the other.  This paper focuses on the case in which one ONB $\Gamma$ is the original Jorgensen-Pedersen ONB for the Cantor measure $\mu_{\frac{1}{4}}$ and the other ONB is is $5\Gamma$.  The main theorem of the paper states that the corresponding operator $U$ is ergodic in the sense that only the constant functions are fixed by $U$.
\end{abstract}

\maketitle

\renewcommand{\baselinestretch}{1.8} \large\normalsize
\section{Introduction}\label{Sec:Intro}

Infinite Bernoulli convolutions are special cases of affine self-similarity systems, also called iterated function systems (IFSs). Thus IFS measures generalize distributions of Bernoulli convolutions (see Section \ref{Subsec:History} for details).  Bernoulli convolutions in turn generalize Cantor measures.  For over a decade, it has been known that a subclass of IFS measures $\mu$ have associated Fourier bases for $L^2(\mu)$ \cite{JoPe98}.    If $L^2(\mu)$ does have a Fourier ONB with Fourier frequencies $\Gamma \subset \mathbb{R}$, we then say that $(\mu, \Gamma)$ is a \textbf{spectral pair}.  In the case that a set of Fourier frequencies exist for $L^2(\mu)$, we say $\Gamma$ is a Fourier \textbf{dual} set for $\mu$ or that $\Gamma$ is a \textbf{spectrum} for $\mu$; we say $\mu$ is a \textbf{spectral measure}.  The goal of this paper is to examine the operator $U$ which scales one spectrum into another spectrum.  We observe how the intrinsic scaling (by $4$) which arises in our set $\Gamma$ interacts with the spectral scaling (to $5\Gamma$) that defines $U$.  We call $U$ an operator-fractal due to its self-similarity, which is described in detail in \cite{JKS11}.

Our main theorem is Theorem \ref{Thm:U_Ergodic}, which states that the only functions which are fixed by $U$ are the constant functions---in other words, $U$ is an ergodic operator in the sense of Halmos \cite{Hal56}.  The background, techniques, and structures necessary to prove this theorem are developed carefully in Sections \ref{Sec:SpecThm} and \ref{Sec:Measures}.  Later, in Section \ref{Sec:MPT}, we explore the interaction of the respective scalings by $4$ and by $5$  (see, for example, Corollary \ref{Cor:Meas}).
 
 The duality results in \cite{JoPe98} can be highly non-intuitive.  For example, when the scaling factor is $\frac13$---that is, $\mu_{\frac13}$ is the Cantor-Bernoulli measure for the omitted third Cantor set construction---there is no Fourier basis.  In other words, there is no Fourier series representation in $L^2(\mu_{\frac13})$.  In fact, there can be at most two orthogonal Fourier frequencies in $L^2(\mu_{\frac13})$  \cite{JoPe98}. But if we modify the Cantor-Bernoulli construction, using scale $\frac14$, as opposed to $\frac13$, then the authors of \cite{JoPe98} proved that a Fourier basis does exist in $L^2(\mu_{\frac14})$.  They showed much more: \textit{each} of the Cantor-Bernoulli measures $\mu_{\frac{1}{2n}}$ with $n\in\N$ has a Fourier basis. For each of these measures, there is a canonical choice for a Fourier dual set $\Gamma$.
  
We consider here a particular additional symmetry relation for the subclass of Cantor-Bernoulli measures that form spectral pairs. Starting with a spectral pair $(\mu, \Gamma)$, we consider an action which scales the set $\Gamma$. In the special case of $\mu_{\frac14}$, we scale $\Gamma$ by $5$.  Scaling by $5$ induces a natural unitary operator $U$ in $L^2(\mu_{\frac14})$, and we study the spectral-theoretic properties of $U$.  

The measure $\mu_{\frac14}$ and its support $X_{\frac14}$ admit the similarity scaling laws shown in Equations \eqref{Eqn:Attr} and \eqref{Eqn:Inv}---scaling by $\frac14$---which we call affine scaling in the \textbf{small}.  The canonical construction of the dual set $\Gamma$ of Fourier frequencies in \cite{JoPe98} uses scaling by powers of $4$ in the \textbf{large}.  Elements of $L^2(\mu_{\frac14})$ are lacunary Fourier series, with the lacunary Fourier bases involving powers of $4$ (see Equation \eqref{eqn:onb} below).

\subsection{Bernoulli convolution measures:  a brief discussion}\label{Subsec:History}The Bernoulli convolution measure with scaling factor $\lambda$, the measure $\mu_{\lambda}$, can be defined in several equivalent ways.  Here, we will describe a probabilistic method and and IFS method to obtain the measure $\mu_{\lambda}$.

In probability theory, one can define the measure $\mu_{\lambda}$ with the distribution of a random variable $\mathbf{Y}_{\boldsymbol{\lambda}}$.  For each $k\in\N$, let 
\[Y_k: \prod_{k=1}^{\infty} \{-1, 1\}\rightarrow \{-1, 1\}\]
be defined by
\begin{equation}
Y_k(\omega_1, \omega_2, \omega_3, \ldots) = \omega_k.
\end{equation}

\begin{lemma}
Define $\mathbf{Y}_{\boldsymbol{\lambda}}$ by
\begin{equation}
\mathbf{Y}_{\boldsymbol{\lambda}}= 
 \sum_{k=1}^{\infty}Y_k\lambda^k.
\end{equation}  Then 
\begin{equation}
\mathbb{E}_{\lambda}(e^{i \mathbf{Y}_{\boldsymbol{\lambda}} t })
= \prod_{k=1}^{\infty} \cos(\lambda^k t).
\end{equation}
\end{lemma}

\noindent \textbf{Note:  }In older notation, $\mathbf{Y}_{\boldsymbol{\lambda}}$ is sometimes written $\mathbf{Y}_{\boldsymbol{\lambda}}= \sum_{k=1}^{\infty}(\pm 1)_k\lambda^k$.
Whatever notation is used, $Y_k$ or $(\pm 1)_k$ is the outcome of the binary coin-toss where each of the two outcomes, heads $(+1)$ and tails $(-1)$, is equally likely.  These coin-tosses are independent of each other and identically distributed.  

\textbf{Proof:  }If $\mathbb{E}_{\lambda}$ denotes the expectation of the random variable 
$\mathbf{Y}_{\boldsymbol{\lambda}}$, then for all $t\in\R$,
\begin{equation}
\begin{split}\label{Eqn:Expect}
\mathbb{E}_{\lambda}(e^{i \mathbf{Y}_{\boldsymbol{\lambda}} t })
& = \mathbb{E}_{\lambda}(e^{\sum_{k}Y_k\lambda^k}) \\
& = \prod_{k=1}^{\infty} \mathbb{E}_{\lambda}(e^{i Y_k \lambda^k t}),
\end{split}
\end{equation}
where independence of the random variables $Y_k$ is used to obtain the second line in Equation \eqref{Eqn:Expect}.   Because the two outcomes $-1$ and $+1$ are equally likely, we obtain
\begin{equation}\label{Eqn:Expect2}
\begin{split}
\mathbb{E}_{\lambda}(e^{i \mathbf{Y}_{\boldsymbol{\lambda}}(\cdot) t })
& = \prod_{k=1}^{\infty} \Bigl( \frac12 e^{i \lambda^k t} + \frac12 e^{-i \lambda^k t}\Bigr) \\
& = \prod_{k=1}^{\infty} \cos(\lambda^k t).\\
\end{split}
\end{equation}\hfill$\Box$

For more details about random Fourier series and this approach to the measure $\mu_{\lambda}$, see \cite{Kah85} and \cite[Chapter 5]{Jor06}.

Another way to generate the measure $\mu_{\lambda}$ is from an \textbf{iterated function system} (IFS) with two affine maps 
\begin{equation}\label{Eqn:Aff_Maps}
\tau_+(x) = \lambda(x+1) \quad \textrm{and}\quad \tau_-(x) = \lambda(x-1).
\end{equation}
By Banach's fixed point theorem, there exists a compact subset of the line, denoted $X_{\lambda}$ and called the \textbf{attractor} of the IFS, which satisfies the invariance property 
\begin{equation}\label{Eqn:Attr}
X_{\lambda} = \tau_+(X_{\lambda}) \cup \tau_-(X_{\lambda}).
\end{equation}
Hutchinson proved that there exists a unique measure $\mu_{\lambda}$  corresponding to the IFS \eqref{Eqn:Aff_Maps}, which is supported on $X_{\lambda}$ and is invariant in the sense that 
\begin{equation}\label{Eqn:Inv}
\mu_{\lambda} = \frac{1}{2} \Bigl(\mu_{\lambda} \circ \tau_+^{-1}\Bigr) + \frac{1}{2} \Bigl(\mu_{\lambda} \circ \tau_-^{-1}\Bigr)
\end{equation}
\cite[Theorems 3.3(3) and 4.4(1)]{Hut81}.
The property in Equation \eqref{Eqn:Inv} defines the measure $\mu_{\lambda}$ and can be used to compute its Fourier transform.  The Fourier transform of $\mu_{\lambda}$ is precisely the same function we saw in Equation \eqref{Eqn:Expect2}:
\begin{equation}
\widehat{\mu}_{\lambda}(t) = \prod_{k=1}^{\infty} \cos(\lambda^k t).
\end{equation}

Bernoulli convolution measures have been studied in various settings, long before IFS theory was developed.  Some of the earliest papers on Bernoulli convolution measures date to the 1930s and work with an infinite convolution definition for $\mu_{\lambda}$; they are \cite{JW35, KW35, W35, E39}.  The history of Bernoulli convolutions up to 1998 is detailed in \cite{PSS98}.

\subsection{Notation, terminology, and summary of results}\label{Subsec:Notation}

We will use the notation $e_t(\cdot)$ to denote the complex exponential function $e^{2\pi i t (\cdot)}$.  Given a set $\Gamma \subseteq \mathbb{R}$, we denote by $E(\Gamma)$ the set $\{e_{\gamma}\,:\, \gamma \in \Gamma\}$.    Throughout, we fix $\lambda = \frac14$ in Equations \eqref{Eqn:Expect2} and \eqref{Eqn:Inv} and work exclusively with the Bernoulli convolution measure $\mu_{\frac14}$, which we call $\mu$.  We will work with the set $\Gamma$ originally defined by \cite{JoPe98} as   
\begin{equation}\label{eqn:onb} 
\begin{split}
\Gamma 
& = \Biggl\{ \sum_{i=0}^m a_i 4^i\,:\, a_i \in \{0,1\}, m \textrm{ finite} \Biggr\}\\
& = \{0, 1, 4, 5, 16, 17, 20, 21, 64, 65, \ldots\}.
\end{split}
\end{equation} 
Jorgensen and Pedersen  showed that $\Gamma$ is a spectrum for $\mu$---that is,  the set of exponential functions $E(\Gamma)$ is an orthonormal basis for $L^2(\mu)$ \cite[Theorem 5.6 and Corollary 5.9]{JoPe98}. 

It is known that other scaling symmetries are possible in $L^2(\mu)$; examples are given in \cite{LaWa02, DuJo09d, JKS11b}.  In particular,  Dutkay and Jorgensen have shown that the ONB property is preserved under scaling by powers of $5$---that is, for each $n\in \N$, each scaled set $5^n\Gamma$ is also an exponential ONB for $L^2(\mu)$  \cite[Proposition 5.1]{DuJo09d}.  This result may be counterintuitive since the resulting scaled set \eqref{Eqn:5Gamma} of Fourier frequencies  appears quite ``thin''.  In this paper, we will restrict our attention to the case $n=1$:
\begin{equation}\label{Eqn:5Gamma}
5\Gamma = \{0, 5, 20, 25, 80, 85, 100, 105, 320, \ldots\}.
\end{equation} 
 
The $5$--scaling property for the ONB \eqref{eqn:onb} induces a unitary operator 
$U$ in $L^2(\mu)$, as given in the next definition.
\begin{definition} Define the operator $U$ on the orthonormal basis $E(\Gamma)$ by
\begin{equation}\label{eqn:U}   
U(e_{\gamma}) := e_{5\gamma}. 
\end{equation}   
\end{definition}
In \cite{JKS11}, we gave operators such as $U$ the name \textit{operator-fractals} due to the self-similarity they exhibit.  Due to this self-similar structure, the spectral representation and the spectral resolution for $U$ are surprisingly subtle.  Despite this, we are able to establish ergodic and spectral-theoretic properties of the unitary operator $U$.  Our main theorems are Theorem \ref{Thm:Hv=Kv}, Theorem \ref{Prop:R_v^*}, and Theorem \ref{Thm:U_Ergodic}.  
\begin{itemize}
\item In Theorem \ref{Thm:Hv=Kv}, we establish the correspondence between Hilbert spaces associated with real measures $m_v$  arising from projection-valued measures and $U$-cyclic subspaces of $L^2(\mu_{\frac14})$.
\item We find an explicit formula for the adjoint of the intertwining operator relating the $U$-cyclic subspace containing $v\in L^2(\mu_{\frac14})$ and $L^2(m_v)$ in Theorem \ref{Prop:R_v^*} . 
\item In our major theorem, Theorem \ref{Thm:U_Ergodic}, we prove that the only functions fixed by $U$ are constant functions---that is, $U$ is an ergodic operator.
\end{itemize}

\subsection{Organization of the paper}
We begin in Section \ref{Sec:Intro} with a background discussion of Fourier bases on Cantor measures and motivate our interest in the operator-fractal $U$.  The proofs of our main theorems rely on projection-valued measures and the spectral theorem for unitary operators, for which we provide a brief background in Section \ref{Sec:SpecThm}.  In Section \ref{Sec:Measures}, we examine $U$-cyclic subspaces of $L^2(\mu)$ in detail using Nelson's theory of $\sigma$-classes.  The material in these sections is a blend of standard theorems, known results with new proofs, and some new results which lead us to our main theorem.   We prove the main theorem --- the ergodicity of $U$ --- in Section \ref{Sec:Spectral}.  In Section \ref{Sec:MPT}, we explore various aspects of the relationships of the scaling factors $(\times 4)$ and $(\times 5)$ inherent in the operator $U$. 

\subsection{Motivation for the study of the operator $U$}\label{Subsec:Motiv}

Equations \eqref{eqn:onb} and \eqref{Eqn:5Gamma} show that $5\Gamma$ is not contained in $\Gamma$, so it would be surprising if $U$ behaved well with respect to iteration---and in fact, it does not.  

\begin{proposition}
The formula $U^ke_{\gamma} = e_{5^k \gamma}$ does not hold in general.  
\end{proposition}
\textbf{Proof:  }It is sufficient to prove inequality for a specific example: consider the case $\gamma = 1$ and $k = 3$.  We have $U(e_1) = e_{5}$, and since $5\in \Gamma$, $U^2(e_{1}) = e_{25}$.  However, $25\not\in\Gamma$, so we expand $e_{25}$ in terms of $E(\Gamma)$ to compute $U(e_{25})$:
\begin{equation}\label{Eqn:U^3e_1}
\begin{split}
U(e_{25})
& = U\Bigl(\sum_{\gamma\in\Gamma} \widehat{\mu}_{\frac14}(25 - \gamma)e_{\gamma} \Bigr)
 = \sum_{\gamma\in\Gamma}\widehat{\mu}_{\frac14}(25 - \gamma)e_{5\gamma}\\
& = \sum_{\gamma\in\Gamma}\widehat{\mu}_{\frac14}(25 - \gamma)
       \Bigl(\sum_{\xi\in\Gamma} \widehat{\mu}_{\frac14}(5\gamma-\xi)e_{\xi}\Bigr)\\
& = \sum_{\xi, \gamma\in\Gamma} \widehat{\mu}_{\frac14}(25 - \gamma)\widehat{\mu}_{\frac14}(5\gamma-\xi)e_{\xi}.
\end{split}
\end{equation}
On the other hand,
\begin{equation}\label{Eqn:e_125}
e_{125} = \sum_{\xi\in\Gamma} \widehat{\mu}_{\frac14}(125 - \xi)e_{\xi}.
\end{equation}
Now compare the $\xi=5$ term in Equations \eqref{Eqn:U^3e_1} and \eqref{Eqn:e_125}.  In Equation \eqref{Eqn:e_125}, the coefficient of $e_5$ is $\widehat{\mu}_{\frac14}(120) =  \widehat{\mu}_{\frac14}(30)\approx 0.50$.  In Equation \eqref{Eqn:U^3e_1}, the coefficient of $e_5$ is 
\[
\sum_{\gamma\in\Gamma}  \widehat{\mu}_{\frac14}(25 - \gamma)\widehat{\mu}_{\frac14}(5\gamma-5)
\approx 0.58.
\]
The approximations were made with $512$ terms of $\Gamma(\frac14)$ in \textit{Mathematica}.

\hfill$\Box$

\noindent\textbf{Corollary.  }\textit{$U$ is not implemented by a transformation of the form $U(f) = f\circ\tau$ where $\tau(x) = 5x \pmod{1}$.}
  
In Section \ref{Sec:MPT}, we'll see that the operator $U$ cannot be spatially implemented by any point transformation.  The distinction between the behavior of unitary operators which are implemented by such a transformation $\tau$ and the behavior of the unitary operator $U$  is one of the motivations for why we study $U$ in detail.  One of our main theorems, Theorem \ref{Thm:U_Ergodic}, states that the only functions fixed by $U$ are the constant functions.  While our unitary operator $U$ is not spatially implemented, we can still form Cesaro means of its iterations, and one of the corollaries of Theorem \ref{Thm:U_Ergodic} is an application of the von Neumann ergodic theorem in Section \ref{Sec:MPT}.  

Another motivation comes from the relationship $U$ has with the representation of the Cuntz algebra $\mathcal{O}_2$ which is realized by the two operators
\[ S_0(e_{\gamma}) = e_{4\gamma} \textrm{ and } S_1(e_{\gamma}) = e_{4\gamma+1}\]
defined on the ONB $E(\Gamma)$.
The operator $U$ commutes with $S_0$ but does not commute with $S_1$.  The fact that $U$ does not commute with $S_1$ makes its spectral theory harder to understand, but the commuting with $S_0$ gives us a foothold into its spectral theory.  The relationship between $U$ and operators forming the representation of $\mathcal{O}_2$ is studied in detail in \cite{JKS11}.

We make a preliminary observation about how $U$ scales elements of the ONB $E(\Gamma)$. 
\begin{lemma}
Suppose $\gamma\in\Gamma$ and $\lambda\in \T$ are such that
\begin{equation}
Ue_{\gamma} = \lambda e_{\gamma}\in L^2(\mu_{\frac14}).
\end{equation}
Then $\gamma = 0$ and $\lambda = 1$.
\end{lemma}


\textbf{Proof:  }Suppose $\gamma\in\Gamma\backslash\{0\}$.  If $Ue_{\gamma} = \lambda e_{\gamma}$, then
\begin{equation}
\begin{split}
0
& = \|Ue_{\gamma} - \lambda e_{\gamma}\|^2_{L^2(\mu_{\frac14})}\\
& = \|e_{5\gamma} - \lambda e_{\gamma}\|^2_{L^2(\mu_{\frac14})}\\
& =  \|e_{4\gamma} - \lambda e_{0}\|^2_{L^2(\mu_{\frac14})} = 2
\end{split}
\end{equation}
since $e_0$ and $e_{4\gamma}$ are distinct elements of the ONB $E(\Gamma)$.  Therefore we have a contradiction.\hfill$\Box$

\subsection{Recent developments and associated literature}\label{Subsec:Recent}

The paper which started much of the work considered here is \cite{JoPe98}.  Since then, a large literature on duality and spectral theory for affine dynamical systems has evolved.  Here, we point out just a few of the most recent developments in the field.  First, the papers of J.-L.\! Li  study orthogonal exponential functions with respect to invariant measures \cite{Li09, Li10a, Li10b}; the papers \cite{AK09}, \cite{XZ08}, \cite{HL08}, and \cite{JKS08} also fit into this framework.  The work of Dutkay and Jorgensen and their coauthors, some of which has already been mentioned centers on Fourier duality:  \cite{DHJ09, DHS09, DuJo09a, DuJo09b, DuJo09c, DJP09}.   Spectral measures for affine IFSs are also studied in the works \cite{FeWa05, LaWa06, FeWa09}.   The relationship of wavelets and frames to self-similar measures is explored in \cite{BoKe10, DHSW11}.   The works of Gabardo and his coauthors are also highly relevant:  \cite{GN98, Gab00, GY06, GY07}.  

\section{The spectral theorem and some of its consequences}\label{Sec:SpecThm}

Starting with the spectral pair $(\mu_{\frac14}, \Gamma)$, where $\Gamma$ is given in \eqref{eqn:onb}, we study a unitary operator $U$ in $L^2(\mu)$ corresponding to a scaling of $\Gamma$ by $5$ in detail.  In order to understand $U$, we ask for its spectrum. Recalling that the projection-valued measure for $U$ is generated by scalar measures in each of the $U$-cyclic subspaces in $L^2(\mu)$, we are faced with some delicate issues from spectral theory.  In particular, there are properties of the cyclic subspaces that demand attention. In fact, in Section \ref{Sec:Measures} below, we prove a characterization theorem which may be of independent interest in a more general framework.

In both Sections \ref{Sec:Measures} and \ref{Sec:Spectral}, we depend heavily on the spectral theorem for unitary operators and some of its consequences.  We collect the necessary results in this section for easy reference; our primary resources are the books by Baggett \cite[Chapters IX and X]{Bag92}, Dunford and Schwartz \cite[Chapter X]{DunSch63}, and Nelson \cite[Chapter 6]{Nel69}.

\subsection{The spectral theorem}\label{Subsec:SpecThmOnly}

We start with the definition of a projection-valued measure.
\begin{definition}\label{Defn:pvm} \cite[p.~165]{Bag92}
Given a set $S$ and a $\sigma$-algebra $\mathcal{B}$ of subsets of $S$, and given a separable Hilbert space $\hs$, then a mapping $A \mapsto p_A$ from $\mathcal{B}$ to the projections on $\hs$ is called a \textbf{projection-valued measure (p.v.m.)} if 
\begin{enumerate} 
\item $p_S = I$ and $p_{\emptyset} = 0$.
\item If $\{A_k\}_{k=1}^{\infty}$ is a disjoint collection of sets from $\mathcal{B}$, then $\{p_{A_k}\}_{k=1}^{\infty}$ is a collection of orthogonal projections (i.e. $p_{A_k}p_{A_j} = 0$ in $\hs$ for $k \neq j$) and \[p_{\cup_k A_k} = \sum_k p_{A_k}.\]  
\end{enumerate}
\end{definition}  

In our setting, the set $S$ in Definition \ref{Defn:pvm} will be the circle $\mathbb{T}$, and the Hilbert space $\hs$ is $L^2(\mu)$.  Although the spectral theorem applies in more generality to normal bounded operators, we only use the spectral theorem for the unitary operator $U$, so we state that version here.  
\begin{theorem}\label{Thm:Spectral}\rm(\textbf{The Spectral Theorem for Unitary Operators})\\
\cite[Theorem 10.10, p.~200]{Bag92} \it
Let $U$ be a unitary operator on $\hs$.  Then there exists a unique Borel p.v.m.~ $E^{U}\!$ on the Borel space $(\mathbb{T}, \mathcal{B})$ such that 
\begin{equation}\label{Eqn:SpecThmInt}
U = \int_{\sigma(U)} z\, \rm{d}\it E^U\!(z).
\end{equation} 
The measure $E^U\!$ is supported on the spectrum of $U$, $\sigma(U)$.
\end{theorem}

Next, we recall the \textbf{functional calculus} associated with the spectral theorem.  Given a Borel function $\phi$ on $\T$, we can study the associated operator $\phi(U)$.  The construction of $\phi(U)$ begins with the case where $\phi$ is a polynomial (with both positive and negative powers) and then extends to continuous functions and then to Borel functions.  The next lemma, which holds for $E^U$-essentially bounded functions $\phi:\T\rightarrow\C$, can be extended to suitable Borel functions $\phi$ by Lemma \ref{Thm:Domain}.   See both  \cite[Theorem 10.9]{Bag92} and \cite[Chapter X.2]{DunSch63}, especially Corollaries X.2.8 and X.2.9 and the material between the two corollaries for more information about the following lemma. 
\begin{lemma}\label{Thm:DSX28}
Suppose $U$ is a unitary operator on the Hilbert space $\hs$ with associated p.v.m.~ $E^U$\!, so that 
\[ U = \int_{\sigma(U)} z\: E^U\!(dz).\]  
Suppose $\phi, \phi_1, \phi_2:\T\rightarrow\C$ are $E^U\!$-essentially bounded, Borel-measurable functions.  Define 
\begin{equation}\label{Eqn:FunCal}
\pi_{U}(\phi) = \phi(U) = \int_{\sigma(U)} \phi(z) \,E^U\!(dz).
\end{equation}
Then
\begin{enumerate}[(i)]
\item $[\phi(U)]^* = \overline{\phi}(U)$.  In other words, $\pi_{U}$ is a $*$-homomorphism.
\item $\pi_{U}(\phi_1\phi_2) = \pi_U(\phi_1)\pi_U(\phi_2)$, and as a result, the operators $\phi_1(U)$ and $\phi_2(U)$ commute.
\item If $\phi(z) \equiv 1$, then $\phi(U)$ is the identity operator.
\item The operator $\phi(U)$ is bounded.
\end{enumerate}
\end{lemma}
We note that the converse of \textit{(iv)} is true as well:  if $\phi(U)$ is bounded, then the function $\phi$ is $E^U$\!-essentially bounded.  Finally, Lemma \ref{Thm:DSX28} is also true for normal operators $N$, with $\T$ being replaced by $\C$.

\subsection{Real Borel measures and the operators $\phi(U)$}\label{Subsec:RealBorel}
For each vector $v\in \hs$, there exists a real-valued Borel measure $m_v$ supported on $\mathbb{T}$ such that 
\begin{equation}\label{Eqn:m_v}
m_v(A) = \langle E^U\!(A)v, v \rangle_{\hs},
\end{equation}
where $E^U\!(A)$ is the projection  $\int_{\sigma(U)} \chi_A(z) E^U\!(dz)$.  When $v$ is a unit vector, note that $m_v$ is a probability measure \cite[(2), p.~ 302]{Rud73}.   

\begin{remark} There is also alternative notation for $m_v(A)$ which emphasizes the fact that 
\[
E^U =(E^U)^* = (E^U)^2.\]
We write 
\[
m_v(A) = \langle E^U\!(A)v, E^U\!(A)v\rangle_{\hs} = \|E^{U}\!(A)v\|_{\hs}^2.
\]
\end{remark}

There is an important isometric connection between operators of the form $\phi(U)$ and the measures $m_v$, which we state as the next lemma. 
\begin{lemma}\label{Lemma:Isom}\rm 
\textbf{(Corollary X.2.9, [DS63])}
\it
Suppose $U$ is a unitary operator on the Hilbert space $\hs$ with associated p.v.m.~ $E^U$\!.  Let $m_v$ be the Borel measure on $\hs$ defined in Equation \eqref{Eqn:m_v}.  Suppose $\phi:\T\rightarrow\C$ is an $E^U$\!-essentially bounded, Borel-measurable function.  Then
\begin{equation}\label{Eqn:m_v_Isom}
\| \phi(U)v\|^2_{\hs} = \int_{\sigma(U)} |\phi(z)|^2 \:\rm d\it m_v(z).
\end{equation}
\end{lemma}
\noindent For any $m_v$-integrable function $\phi$ on $\mathbb{T}$, 
\begin{equation}\label{Eqn:m_v_b}
\begin{split}
\int_{\mathbb{T}} \phi(z) \dd m_v(z) 
& = \Bigl\langle \int_{\mathbb{T}} \phi(z) E^U\!(dz) v, v \Bigr\rangle_{\hs}\\
& = \langle \phi(U) v, v \rangle_{\hs}.
\end{split} 
\end{equation}

We noted earlier that $\phi(U)$ is a bounded operator if and only if $\phi$ is $E^U\!$-essentially bounded.  However, $\phi(U)$ can be a well-defined unbounded operator for some unbounded Borel functions $\phi:\T\rightarrow\C$.   In the case that $\phi(U)$ is an unbounded operator, we need to be especially vigilant about the domain of $\phi(U)$.  When $\phi(U)$ is a well-defined unbounded operator, the usual formulas discussed in the bounded case carry over.  By fixing $U$, one obtains an algebra of operators from the Borel functions on $\T$:
\[A_U:=\{ \phi(U) :  \phi \textrm{ a Borel function on } \T\}.\]
 Specifically,  $A_U$ turns into a commutative algebra of (generally unbounded) normal operators, and all the operators in $A_U$ have a common dense domain.  In what follows, we discuss $A_U$ carefully; see also \cite{Ka86, Jo79, Jo80}.
 Specifically, we show that for any Borel function $\phi$, the operator $\phi(U)$ is normal and therefore closed.  Then in Lemma \ref{Thm:Domain}, we show that the domain of $\phi(U)$ is determined by the measures $m_v$.  Once we establish Lemma \ref{Thm:Domain}, the results of Lemmas \ref{Thm:DSX28} and \ref{Lemma:Isom} can be extended to suitable Borel functions $\phi$ and not just $E^U\!$-essentially bounded Borel functions on $\T$.

\subsection{The algebra $A_N$ for a normal operator $N$}

In this section, we work with a normal operator $N$ instead of restricting to the unitary operator $U$.  Although the following result is known, we present an approach via a theorem of Stone \cite[Theorem 9]{Stone}.

\begin{lemma}\label{Thm:Domain}
Suppose $\phi$ is a Borel measurable function on $\T$ and $N$ is a normal operator on the Hilbert space $\hs$.  Let $\phi(N)$ be the operator defined by
\begin{equation}
\phi(N) = \int_{\T} \phi(z) E^N\!(dz).
\end{equation}
Then $\phi(N)$ is a densely defined operator, and
$v\in \rm{dom}\it(\phi(N))$ if and only if  $\phi\in L^2(m_v)$.  In this case, the isometry in \eqref{Eqn:m_v_Isom} holds:
\begin{equation}
\|\phi(N)v\|^2 = \int_{\T} |\phi(z)|^2 \:\rm{d}\it m_v(z).
\end{equation}
\end{lemma} 
The proof of this lemma is contained in Lemmas \ref{lem_a} through \ref{lem_c} below.

To begin, we review the material from Stone.  For any operator $A$ on a Hilbert space $\hs$, we can refer to the graph of $A$, $G(A)$, as
\[ G(A) = \{ (x,y)\in \hs \oplus\hs:\textrm{ there exists }x \textrm{ such that }y = Ax\}.\] The operator $A$ is called \textbf{closed} when $G(A)$ is closed.  Let $P:\hs\oplus\hs\rightarrow G(A)$ be the self-adjoint, orthogonal projection of $ \hs \oplus\hs$ onto $G(A)$.  Then $P$ has a standard $2\times 2$ operator matrix, called the \textbf{characteristic matrix} of $A$.  We study the elements $P_{i,j}$, $1 \leq i, j\leq 2$ of the characteristic matrix of $P$ for an operator $A$ of the form $\phi(N)$, where $N$ is normal.  Our goal will be to show that $\phi(N)$ is normal using the following theorem of Stone:
\begin{theorem}\label{Stone_Thm}\rm 
\textbf{(Theorem 9, [Sto51], verbatim)}
\it
A two-rowed matrix of bounded linear operators in $\hs$ is the characteristic matrix of a normal operator $A$ if and only if it has the form
\begin{equation}\label{char_mx}
\begin{pmatrix} B & C \\ C^* & I-B \end{pmatrix},
\end{equation}
where $B$ is an invertible self-adjoint operator, and $C$ is a normal operator which commutes with $B$ and satisfies the identity
\[ CC^* = B- B^2.\]  
In terms of this matrix, $A$ and $A^*$ are given by the identities $A= B^{-1}C$, $A^* = B^{-1}C^*$.
\end{theorem}
In the notation used earlier for the operators in the characteristic matrix, $P_{1,1}= B$, $P_{1,2} = C$, $P_{2,1} = C^*$, and $P_{2,2} = I - B$.    Since the domain of $A$, $\rm{dom}\it(A)$, is the first component in $G(A)$, we can characterize $\textrm{dom}(A)$ by
\[\rm{dom}\it(A)=\{ By_1 + Cy_2 : (y_1, y_2)\in\hs\oplus\hs\}.\]
Another way to characterize $\textrm{dom}(A)$ is
\begin{equation}\label{dom}
 \rm{dom}\it(A)=\{ x\in \hs : Cx \in B\hs\};
 \end{equation}
the proof is contained in the proof of Stone's Theorem 9 \cite[p.~169]{Stone}.

Given a Borel function $\phi$, we can now define candidates for the operators $B$ and $C$ in Theorem \ref{Stone_Thm}, so that the characteristic matrix of $\phi(N)$ is given by \eqref{char_mx}.   Recall the $*$-homomorphism $\pi$ defined in Lemma \ref{Thm:DSX28}, where $\T$ is replaced by $\C$ in the case of the normal operator $N$.  
\begin{lemma}Let $\phi$ be a Borel function on $\C$, and let $N$ be a normal operator on $\hs$.  
Set
\[
B = \pi_N\Biggl(\frac{1}{1 + |\phi|^2}\Biggr) \quad\textrm { and } \quad C = \pi_N\Biggl(\frac{\overline{\phi}}{1 + |\phi|^2}\Biggr).
\]
Then $B$ and $C$ satisfy all the conditions in Theorem \ref{Stone_Thm}, and $\phi(N)$ is a normal operator.
\end{lemma}
\begin{proof}  Since $(1 + |\phi|^2)^{-1}$ and $\overline{\phi}(1 + |\phi|^2)^{-1}$ are bounded Borel functions, we can apply Lemma \ref{Thm:DSX28}.  By definition, $B$ is self-adjoint (i).  The operators $B$ and $C$ commute (ii); $C$ and $C^*$ also commute, so $C$ is normal (ii).  It is easy to check that the equation $CC^* = B - B^2$ holds.  Finally, $B$ is one-to-one, so $B$ is invertible.  Therefore $B$ and $C$ satisfy all the conditions of Theorem \ref{Stone_Thm}, and $\phi(N)$ is a normal operator. 
\end{proof} 

\begin{lemma}
If $\phi$ is a Borel function on $\C$, and $N$ is a normal operator on $\hs$, then $\phi(N)$ is closed.
\end{lemma}

\begin{proof}
Recall from Equation \eqref{dom} that the domain of $\phi(N)$ can be characterized in terms of the bounded operators $B$ and $C$.  Suppose $\{x_n\}$ is a sequence belonging to $\textrm{dom}(\phi(N))$ with limit $x\in\hs$.  Suppose that $\phi(N)x_n\rightarrow y\in\hs$.  We need to show that $\phi(N)x = y$.

Since $x_n\in\textrm{dom}(\phi(N))$, there exists $y_n\in\hs$ such that $Cx_n = By_n$ by \eqref{dom}.   In other words, by the last sentence of Theorem \ref{Stone_Thm}, $B^{-1}Cx_n = \phi(N)x_n = y_n$.  But $\phi(N)x_n\rightarrow y$ by hypothesis.  Therefore $y_n\rightarrow y$.  Since $B$ and $C$ are bounded, we know that $Cx_n = By_n$ for all $n$ implies that $Cx = By$.  Therefore $y = B^{-1}Cx = \phi(N)x$, and $\phi(N)$ is closed.
\end{proof}

For each $j\in\N$, define $A_j$ to be the pullback
\[ 
A_j = \{ z\in \C \:|\: |\phi(z)| \leq j\}.
\]
For each $j\in\N$, $\chi_{(A_j)}(z)$ is a bounded function on $\C$, and by Lemma \ref{Thm:DSX28},
\[
E^N\!(A_j) = \chi_{A_j}(N)
\]
is a bounded operator.  Set 
\begin{equation}\label{Eqn:v_j}
v_j  = E^N\!(A_j)v.
\end{equation}

The sequence 
\begin{equation}\label{Eqn:PW}
\{
\chi_{(A_j)}\phi
\}_{j=1}^{\infty}
\end{equation} 
converges pointwise to $\phi$ on $\C$ because $\C = \cup_{j=1}^{\infty} A_j$.

\begin{lemma}\label{lem_a}
Let $N$ be a normal operator on $\hs$, and let $\phi$ be a Borel function on $\C$.  Then the operator $\phi(N)$ is densely defined.
\end{lemma}
\begin{proof}
For each $x\in\hs$ and each $j\in\N$, the vector $E^{N}\!(A_j)x = \chi_{A_j}(N)x\in \textrm{dom}(\phi(N))$.  Furthermore, the set
\[
\{ E^{N}\!(A_j)x \: | \: x\in\hs \textrm{ and } j\in\N\}
\]
is dense in $\hs$ because the projections $E^{N}\!(A_j)$ tend to the identity (each $x\in\hs$ is the limit of the sequence $\{E^{N}\!(A_j)x\}$).  Therefore $\phi(N)$ is densely defined.
\end{proof}

\begin{lemma}
Let $N$ be a normal operator on $\hs$.  Suppose $v\in \rm{dom}\it(\phi(N))$.  Then $\phi\in L^2(m_v)$, where $m_v$ is the measure defined in Equation \eqref{Eqn:m_v}.
\end{lemma}
\begin{proof} We show that if $v\in \textrm{dom}(\phi(N))$, then $\phi\in L^2(m_v)$. Because $v\in \textrm{dom}(\phi(N))$ and $\phi(N)$ is closed, the sequence $\{\phi(U)v_j\}_{j=1}^{\infty}$ where $v_j$ is defined in \eqref{Eqn:v_j} converges to an element $w\in\hs$.   Since $\chi_{A_j}\phi$ converges pointwise to $\phi$, apply Fatou's Lemma \cite[Theorem 1.28, p. 23]{Rud87}: 
\begin{equation}
\begin{split}
\int  |\phi|^2  \dd m_v
& = \int \liminf_{j\rightarrow\infty} |\chi_{A_j}\phi|^2 \dd m_v\\
& \leq\liminf_{j\rightarrow\infty} \int  |\chi_{A_j}\phi|^2 \dd m_v\\
& = \liminf_{j\rightarrow\infty}\: \langle \overline{\chi_{A_j}(U)\phi(U)} \chi_{A_j}(U)\phi(U)v, v\rangle_{\hs}\\
& = \liminf_{j\rightarrow\infty}\:\|\phi(U) v_j\|^2_{\hs}\\
& = \lim_{j\rightarrow\infty} \|\phi(U) v_j\|^2_{\hs} = \|w\|^2_{\hs} < \infty.
\end{split}
\end{equation}
Therefore $\phi\in L^2(m_v)$.

Alternately, one could take the supremum over the integrals 
\[
\int_{A_j} |\phi|^2 \dd m_v
\]
to establish that $\phi\in L^2(m_v)$.
\end{proof}

\begin{lemma}\label{lem_c}
Let $N$ be a normal operator on $\hs$.  Suppose $\phi\in L^2(m_v)$, where $m_v$ is the measure defined in Equation \eqref{Eqn:m_v}.  Then $v\in \rm{dom}\it(\phi(N))$. 
\end{lemma}
\begin{proof}
We show that if $\phi\in L^2(m_v)$, then $v\in \textrm{dom}(\phi(N))$.  
To establish that the vector $v$ belongs to the domain of $\phi(N)$, we can show that the sequence $(v_j, \phi(N)v_j)\subseteq \hs\oplus \hs$ has a limit in $\hs\oplus \hs$, where $\{v_j\}$ is defined in \eqref{Eqn:v_j}.  We know that $v_j\rightarrow v$, so we need to consider the second component.  Since $\phi\in L^2(m_v)$, the pointwise limit in \eqref{Eqn:PW} is a limit in $L^2(m_v)$ as well.  In other words, the sequence $\{ \chi_{A_j} \phi\}$ is Cauchy:
\[
\lim_{j, k\rightarrow\infty} \| \chi_{A_j} \phi - \chi_{A_k}\phi \|^2_{L^2(m_v)} =0.
\]
But $\chi_{A_j} \phi - \chi_{A_k}\phi$ is a bounded function on $\C$, so we can apply Lemma \ref{Thm:DSX28} (iii):
\begin{equation}
\begin{split}
& \| \chi_{A_j} \phi - \chi_{A_k}\phi \|^2_{L^2(m_v)} \\
& = \|  \chi_{A_j}(N) \phi(N)v - \chi_{A_k}(N)\phi(N)v \|^2_{\hs}\\
& = \| \phi(N)v_j - \phi(N)v_k\|^2_{\hs}
\end{split}
\end{equation}
which implies that $\{\phi(N)v_j\}$ is Cauchy in $\hs$.
\end{proof}

\section{The unitary operator $U$}\label{Sec:Measures}

In this section we introduce the unitary operator $U$ in $\hs = L^2(\mu_{\frac14})$ defined from the 5-scaled ONB $5\Gamma$ mentioned in Section \ref{Sec:Intro}.  We prove that $U$ has a number of intriguing affine self-similarity properties.  Its cyclic subspaces are studied in Theorems \ref{Thm:Hv=Kv} and \ref{Prop:R_v^*}.  A key component in the proof of our main result, Theorem \ref{Thm:U_Ergodic}, is the correspondence between the $U$-cyclic subspaces of $L^2(\mu_{\frac14})$ and the Hilbert spaces $L^2(m_v)$ with respect to the scalar measures generated by $U$.

\subsection{Cyclic subspaces associated with $U$}\label{Subsec:Cyclic} 
The definition of a cyclic subspace $H(v)$ given in Definition \ref{Defn:Cyc_Subsp} is taken from Nelson, who uses the same definition with positive powers of $U$ only \cite[p. 68-69]{Nel69}.  This is because Nelson is working with self-adjoint operators.  Because the operator $U$ defined in Equation (\ref{eqn:U}) is not self-adjoint, we add $U^*$-invariance to Definition \ref{Defn:Cyc_Subsp}. 
\begin{definition}\label{Defn:Cyc_Subsp}
Let $v\in L^2(\mu_{\frac{1}{4}})$.  Define $H(v)$ to be the smallest closed subspace of $L^2(\mu_{\frac{1}{4}})$ which contains $v$ and is invariant under both $U$ and $U^*=U^{-1}$.  We call $H(v)$ the $\mathbf{U}$-\textbf{cyclic subspace} for $v$.\end{definition}

\noindent A slightly more useful version of Definition \ref{Defn:Cyc_Subsp} is the following:

\medskip

\noindent \textbf{Definition \ref{Defn:Cyc_Subsp}, restated.   } Let $v\in L^2(\mu_{\frac{1}{4}})$.  The \textbf{U-cyclic subspace} $H(v)$ is the intersection of all subspaces $K\subset L^2(\mu_{\frac14})$ such that 
\begin{enumerate}[(a)]
\item $K$ is closed
\item $v\in K$
\item $UK\subset K$
\item $U^*K\subset K$.
\end{enumerate}

\medskip

There are other ways to describe the $U$-cyclic subspace $H(v)$.  A simple one is described in Lemma \ref{Lemma:H(v)_and_U^k}.
\begin{lemma}\label{Lemma:H(v)_and_U^k}
Let $v\in L^2(\mu_{\frac{1}{4}})$ with $\|v\|=1$.  Then
\begin{equation}
H(v)  = \overline{\rm{span}}_{L^2(\mu_{\frac{1}{4}})}\it\{U^k v \: |\: k \in \mathbb{Z} \}.
\end{equation}
In other words, if $\phi(z)$ is the polynomial
\begin{equation}\label{Eqn:PhiPoly}
\phi(z) = \sum_{k=-N}^N c_k z^k
\end{equation}
where $c_k\in\C$, $k = -N, \ldots, N$ and $z\in \T$, then the vectors $\phi(U)v$ are dense in $H(v)$.
\end{lemma}

\noindent\textbf{Proof:  }Since $H(v)$ contains $v$, and since $H(v)$ is invariant under $U$, we have
\[ Uv \in U(H(v)) = H(v).\]
By induction, all vectors of the form $U^nv$ where $n\in\N$ belong to $H(v)$.  Similarly,  $U^*v\in H(v)$, so all vectors of the form $U^{-n}v$ where $n\in\N$ also belong to $H(v)$.  Since $H(v)$ is closed, the subspace  $\overline{\rm{span}}\it\{U^k \: |\: k \in \mathbb{Z} \}$ is contained in $H(v)$.  However, $\overline{\rm{span}}\it\{U^k v\: |\: k \in \mathbb{Z} \}$ contains $v$ and is invariant under both $U$ and $U^*$, and $\overline{\rm{span}}\it\{U^k v \: |\: k \in \mathbb{Z} \}$ cannot be a proper subspace of $H(v)$ since $H(v)$ is the smallest $U$,$U^*$-invariant subspace containing $v$.

We note that since functions $\phi$ as in Equation (\ref{Eqn:PhiPoly}) are continuous on $\T$, the functions $\phi$ are bounded and therefore define bounded operators on $L^2(\mu_{\frac{1}{4}})$, so in particular $\phi(U)v$ is defined as in Lemma \ref{Thm:DSX28}.\hfill$\Box$

The characterizations of $H(v)$ given so far do not involve a measure.  We will establish a different characterization of $H(v)$ which directly connects $H(v)$ to the space $L^2(m_v)$.

\begin{definition}\label{Defn:Kv} Suppose $v\in L^2(\mu_{\frac{1}{4}})$ with $\|v\|=1$.  Define $K(v)\subset L^2(\mu_{\frac{1}{4}})$ as follows:
\begin{equation}\label{Eqn:Kv}
K(v) := \{ \phi(U)v \:|\: \phi\in L^2(m_v)\}.
\end{equation}
\end{definition}
Recall that $\phi\in L^2(m_v)$ if and only if $v$ belongs to the domain of $\phi(U)$ by Lemma \ref{Thm:Domain}. 

\begin{theorem}\label{Thm:Hv=Kv} 
Given $v \in L^2(\mu_{\frac14})$ with $\|v\|=1$, there is an isometric isomorphism between the cyclic subspace $H(v)$ and the Hilbert space $L^2(m_v)$.
\end{theorem}

\textbf{Proof:  }
Given $v$, let $K(v)$ be as in Equation \eqref{Eqn:Kv}.  Consider the natural map $\phi(U)v \mapsto \phi$ from $K(v)$ to $L^2(m_v)$.  We find that this map is an isometry: 
\begin{equation}\label{Eqn:Isom}
\begin{split}
 \|\phi\|^2 _{L^2(m_v)}
 & =\int |\phi|^2(z) \dd m_v(z)  \\
 & \underset{(\ref{Eqn:m_v_b})} {=}\Bigr\langle  \Bigr( \int |\phi|^2 \dd E^U \Bigr)v, v \Bigr\rangle_{L^2(m_v)}\\
 & = \langle |\phi|^2(U)v, v \rangle_{L^2(\mu_{\frac{1}{4}})}\\
 & = \langle \phi(U)v,\phi(U)v \rangle_{L^2(\mu_{\frac{1}{4}})} \\
 & = \|\phi(U)v\|^2_{L^2(\mu_{\frac{1}{4}})}.
\end{split}
\end{equation}
Thus $\phi(U)v \mapsto \phi$ is injective, and hence is an isometric isomorphism between $K(v)$ and $L^2(m_v)$.

Next, we will establish that the space $K(v)$ defined in Definition \ref{Defn:Kv}  satisfies conditions (a) -- (d) in the second version of Definition \ref{Defn:Cyc_Subsp}:
\begin{enumerate}[(a)]
\item $K(v)$ is closed
\item $v\in K(v)$
\item $UK(v)\subset K(v)$
\item $U^*K(v)\subset K(v)$.
\end{enumerate}
This will prove that  $H(v) \subseteq K(v)$; we continue on to prove that $H(v) = K(v)$, which will complete the proof. 

For (a), we note that $L^2(m_v)$ is closed and complete.  Since there is an isometry from $L^2(m_v)$ to $K(v)$, we know $K(v)$ is closed.  

Since $m_v$ is a probability measure, the constant function $\mathbf{1}$ defined by $\mathbf{1}(z) = 1$ for all $z\in \T$ belongs to $L^2(m_v)$.  Therefore $\mathbf{1}(U)$ is the identity operator $I$ and $\mathbf{1}(U)v = Iv = v$, so $v\in K_v$, which establishes (b).  

Next, we want to show that $K(v)$ is invariant under $U$.  We will take advantage of polynomials in $U$ and $U^*$, just like in Lemma \ref{Lemma:H(v)_and_U^k}.  Let $\frak{A}$ be the algebra of functions on $\T$:
\begin{equation}\label{Eqn:A}
\frak{A}: = \Bigl\{p(z) = \sum_{-N}^N c_k z^k \Bigr\}.
\end{equation}
 Define the associated space $K^{\textrm{pol}}(v)$ as follows: 
\begin{equation}\label{Eqn:Kpol}
K^{\textrm{pol}}(v): = \{ p(U)v \: | \: p\in \frak{A}\}.
\end{equation}
We will show that $K^{\textrm{pol}}(v)$ satisfies (c) and (d).  Then we will show that $K^{\textrm{pol}}$ is dense in $K(v)$, so $K(v)$ satisfies (c) and (d) as well.  

Define $R_v:  \frak{A}\rightarrow L^2(\mu)$ by
\begin{equation}\label{Eqn:R_v} 
R_v(\phi) = \phi(U)v.
\end{equation}
By Equation (\ref{Eqn:Isom}), $R_v$ is an isometry.

For (c), let $M_z$ denote multiplication by $z$ on functions with domain $\T$.  The following diagram commutes, and as a result, $K^{\textrm{pol}}(v)$ is invariant under $U$ and $U^*$:
\begin{equation}\label{Eqn:CommDiag}
\begin{CD}
\frak{A} @> R_v >>  L^2(\mu)\\
@VV M_z V    @VV  U V\\
\frak{A} @> R_v >>  L^2(\mu)
\end{CD}.
\end{equation}
To see this, let $\phi\in\frak{A}$, where $\phi(z) = \sum_{-N}^N c_k z^k$.  Then following the diagram across the top, we have
\[ R_v(\phi) = \sum_{-N}^N c_k U^kv,\]
and then following the diagram down the right-hand side, we have
\[ UR_v(\phi) = \sum_{-N}^N c_k U^{k+1}v.\]
On the other side,
\[ M_z\phi(z) = \sum_{-N}^N c_k z^{k+1},\]
and $\sum_{-N}^N c_k z^{k+1}\in\frak{A}$.  Applying $R_v$ to $M_z\phi(z)$, we obtain $UR_v(\phi)$.  Therefore, if we set $\psi(z) = z\phi(z)$, we have $\psi\in\frak{A}$, and $U\phi(U)v$ = $\psi(U)v$.  Therefore $K^{\textrm{pol}}(v)$ is invariant under $U$.  

The argument that $K^{\textrm{pol}}(v)$ is invariant under $U^*$ is identical.

Again, let $U$ be the unitary operator defined in Equation (\ref{eqn:U}), let $v\in L^2(\mu_{\frac{1}{4}})$, and let $K(v)$ be the space defined in Equation (\ref{Eqn:Kv}):
\[K(v) = \{ \phi(U)v \:|\: \phi\in L^2(m_v)\}.\]
We claim that $K(v)$ itself (not just $K^{\textrm{pol}}(v)$) is invariant under $U$ and $U^*$. 

Let $\Psi\in K(v)$.  By definition, we can choose $\psi\in L^2(m_v)$ such that
\[ \Psi = \psi(U)v.\]
Instead of working directly with operators of the form $\psi(U)$, we use the Spectral Theorem and the commutative diagram in Equation \eqref{Eqn:CommDiag} to work with functions on $\T$.

Let $\varepsilon > 0$.  First, $m_v$ is a Borel measure, and continuous functions on $\T$ are dense in $L^2(m_v)$ by the Riesz Representation Theorem \cite[Theorem 2.14 p.~ 41 and Theorem 3.14, p.~ 69]{Rud87}.  Therefore, we can choose $\phi\in C(\T)$ so that
\[ \|\psi - \phi\|_{L^2(m_v)} < \frac{\varepsilon}{2}.\]
Second, by the Stone-Weierstrass Theorem, $\frak{A}$ is dense in $C(\T)$ with respect to the $L^{\infty}(\T)$ norm.  Choose $p\in\frak{A}$ so that
\[ \|\phi - p\|_{L^{\infty}(\T)} < \frac{\varepsilon}{4}.\]
Since $m_v(\T) = 1$, $\phi$ and $p$ are also close in the $L^2(m_v)$ norm: 
\[ \|\phi - p\|_{L^2(m_v)} \leq \|\phi - p\|_{L^{\infty}(\T)} < \frac{\varepsilon}{2}.\]
Therefore we can approximate $\psi\in L^2(m_v)$ with a polynomial $p\in \frak{A}$:
\[ \| \psi - p\|_{L^2(m_v)} <  \varepsilon.\]

Finally, the isometry in Equation (\ref{Eqn:Isom}) gives us 
\begin{equation}
\|\psi - p\|_{L^2(m_v)} = \|\Psi - p(U)v\|_{L^2(\mu_{\frac{1}{4}})},
\end{equation}
and since $U$ is unitary,
\[ \|\Psi - p(U)v\|_{L^2(\mu_{\frac{1}{4}})}= \|U\Psi - Up(U)v\|_{L^2(\mu_{\frac{1}{4}})} < \varepsilon.\]
Note that $Up(U)v\in K^{\textrm{pol}}(v)$ by Equation \eqref{Eqn:CommDiag}.  

Since each element of $K(v)$ can be made arbitrarily close to an element of the $U$-invariant space $K^{\textrm{pol}}(v)$, $K(v)$ is also invariant under $U$.  The same argument applies to $U^*$.

We have now established that $K(v)$ satisfies conditions (a) -- (d) in Definition \ref{Defn:Cyc_Subsp}.  Therefore the $U$-cyclic subspace $H(v)$ is contained in $K(v)$.  We claim now that $H(v) = K(v)$.  Let $P_{H(v)}$ be the orthogonal projection onto $H(v)$.  Since $H(v)$ is invariant under $U$ and $U^*$, $P_{H(v)}$ commutes with $U$ and $U^*$ and therefore with all functions of $U$ and $U^*$.   Now, let $\Psi\in K(v)$, and choose $\psi\in L^2(m_v)$ so that $\Psi = \psi(U)v$.  We apply the projection $P_{H(v)}$:
\begin{equation}
P_{H(v)}\psi(U)v
= \psi(U) P_{H(v)} v
= \psi(U)v.
\end{equation}

We have proved that $K(v) = H(v)$, so there is an isometric isomorphism between $H(v)$ and $L^2(m_v)$.   \hfill$\Box$

\subsection{The Hilbert space $\mathscr{H}(\T)$ of $\sigma$-functions}\label{Subsec:LSsigma}

In Theorem \ref{Thm:Hv=Kv}, we showed that every element  $w$ of the cyclic subspace $H(v)$ can be written uniquely in the form $w=\psi(U)v$.   Next, given $w\in H(v)$, we will explicitly compute the corresponding function $\psi$ in Theorem \ref{Thm:Hv=Kv}.  Our result is the following, which is proved in Theorem \ref{Prop:R_v^*}.

\smallskip

\noindent\textbf{Theorem.  }\it  Let $w\in H(v)$, and choose $\psi\in L^2(m_v)$ such that $w = \psi(U)v$.  Then 
\begin{equation}\label{Eqn:RNpsi}
\psi = \sqrt{\frac{dm_w}{dm_v}}.
\end{equation}
\rm

We postpone the formal proof of the proposition immediately above so that we can first explain the techniques used in the proof.  

\subsubsection{Introduction to the Hilbert space $\mathscr{H}(\T)$}  Our main tool in proving Equation \eqref{Eqn:RNpsi} will be the L$\rm\acute{e}$vy-Schwartz Hilbert space $\mathscr{H}(\T)$ of $\sigma$-functions and its associated inner product $\langle \cdot, \cdot\rangle_{\mathscr{H}(\T)}$.  Details about this Hilbert space are developed in Nelson \cite[p.~ 77 ff]{Nel69}.  In particular, the inner product on $\mathscr{H}(\T)$ allows us to move back and forth between the measures $m_w$ and $m_v$ in Equation \eqref{Eqn:RNpsi}.  

Nelson's theory of $\sigma$-functions is developed in conjunction with classifying representations of $C(X)$, where $X$ is a compact Hausdorff space \cite[p. 81]{Nel69}.  His goal is to  write Hilbert spaces as orthogonal direct sums of cyclic subspaces such as those defined in Definition \ref{Defn:Cyc_Subsp}.  \
We translate the work in Nelson to our setting.  First, the compact Hausdorff space $X$ in Nelson is the circle $\T$.  This is because the operator $U$ with which we work is unitary, and its spectrum is contained in $\T$.  The measures with which we work, the measures $m_v$ in Equation \eqref{Eqn:m_v}, are supported on the spectrum of the unitary operator $U$.  

Second, we work with a representation of $C(\T)$ arising from $U$.  Nelson's representation do not necessarily arise from a unitary operator.  The operator $U$ defined in Equation \eqref{eqn:U} induces the representation in Lemma \ref{Thm:DSX28}, and it is precisely the conditions in Lemma \ref{Thm:DSX28} which we need to fit into Nelson's framework.

\subsubsection{The class of $\sigma$-functions and the Nelson isomorphism}\label{Subsec:NIsom}
Second, Nelson develops the theory of $\sigma$-functions on $\T$.  A $\sigma$-function is an equivalence class of pairs of functions and measures as defined below.
\begin{definition}\label{Defn:SigCl}\cite[p. 84]{Nel69}  Suppose $m$ is a Borel measure on $\T$.  A $\mathbf{\sigma}$-\textbf{function} on $\T$ is an equivalence class represented by a pair $(\phi,m)$ where $\phi\in L^2(m)$.  Two pairs $(\phi_1, m_1)$ and $(\phi_2, m_2)$ belong to the same equivalence class if there exists a measure $m$ on $\T$ such that
\begin{enumerate}[(i.)]
\item $ m_1 \ll m$
\item $ m_2 \ll m$
\item $\displaystyle \phi_1 \sqrt{\frac{dm_1}{dm}} = \phi_2\sqrt{\frac{dm_2}{dm}}$ a.e. $m$.
\end{enumerate}
\end{definition}
If we fix $v\in L^2(\mu_{\frac{1}{4}})$ and let $\phi,\psi\in L^2(m_v)$, then the classes $(\phi, m_v)$ and $(\psi, m_v)$ are the same if and only if $\phi=\psi$ a.e. $m_v$---that is, if and only if $\phi =\psi$ in $L^2(m_v)$. 

Nelson proves that Definition \ref{Defn:SigCl} above is independent of the choice of the measure $m$.  Equivalence classes can be added, and, with $m$ as in Definition \ref{Defn:SigCl}, there is an inner product on $\mathscr{H}(\T)$ defined by
\begin{equation}\label{Eqn:NelsonIP}
\langle (\phi_1, m_1), (\phi_2, m_2) \rangle_{\mathscr{H}(\T)} = \int \overline{\phi_1}\phi_2\sqrt{\frac{dm_1}{dm}}\sqrt{\frac{dm_2}{dm}}\:dm,
\end{equation}
which again is independent of $m$ \cite[p. 85]{Nel69}.  Nelson also shows that the vector space $\mathscr{H}(\T)$ is complete, so that $\mathscr{H}(\T)$ is indeed a Hilbert space.  We will use Equation \eqref{Eqn:NelsonIP} to compute the function $\psi$ in Equation \eqref{Eqn:RNpsi}.

Fix $v\in L^2(\mu_{\frac{1}{4}})$ and let $\phi\in L^2(m_v)$.  There is a natural way to define a $\sigma$-function and associate that $\sigma$-function in $\mathscr{H}(\T)$ with an element of of the $U$-cyclic space $H(v)$.  We make the following association, which we call the \textbf{Nelson isomorphism}:
\begin{equation}\label{Eqn:NIsom}
(\phi, m_v) \longleftrightarrow  \phi(U)v.
\end{equation}
The association in Equation \eqref{Eqn:NIsom} is isometric:  by Equation \eqref{Eqn:Isom}, 
\begin{equation*}
\|\phi(U)v\|_{L^2(\mu_{\frac{1}{4}})} = \|\phi\|_{L^2(m_v)}.
\end{equation*}
Since $m_v \ll m_v$, we can apply Equation \eqref{Eqn:NelsonIP} to compute the $\mathscr{H}(\T)$-norm of $(\phi, m_v)$: 
\begin{equation}
\begin{split}
\| (\phi, m_v)\|_{\mathscr{H}}^2
& = \langle (\phi, m_v), (\phi, m_v)\rangle_{\mathscr{H}(\T)}\\
& = \int_{\T} \phi \overline{\phi} \sqrt{\frac{dm_v}{dm_v}} \sqrt{\frac{dm_v}{dm_v}}\:dm_v\\
& = \int_{\T} |\phi|^2 \:dm_v\\
& =  \|\phi\|^2_{L^2(m_v)}\\
& = \|\phi(U)v\|^2_{L^2(\mu_{\frac14})}.
\end{split}
\end{equation}

Theorem \ref{Thm:Hv=Kv} tells us that the isometry above is in fact onto the $U$-cyclic subspace $H(v)$, so Equation \eqref{Eqn:NIsom} defines an isometric isomorphism.

\subsubsection{Absolute continuity and the $U$-cyclic subspace $H(v)$}

We have already established that for each $v$, the cyclic subspace $H(v)$ has two equivalent definitions:  the definition based on minimality given in Definition \ref{Defn:Cyc_Subsp} or the characterization given in Theorem \ref{Thm:Hv=Kv}:
\[ H(v) = \{ \phi(U)v \:|\:\phi\in L^2(m_v)\}.\]
Let $w\in H(v)$, and let $f\in C(\T)$.  Choose $\phi\in L^2(m_v)$ such that $w = \phi(U)v$. Then
\begin{equation}
\begin{split}
\int_{\T} f(z) \:dm_w(z)
& = \langle w, f(U)w \rangle\\
& =\langle \phi(U)v, f(U)\phi(U)v\rangle \\ 
\end{split}
\end{equation}
By Lemma \ref{Thm:DSX28} we know
\[ f(U)\phi(U) = \phi(U)f(U),\]
and we also know that the adjoint of $\phi(U)$ is $\overline{\phi}(U)$.  So,
\begin{equation}\label{Eqn:mwtomv}
\begin{split}
\int_{\T} f(z) \:dm_w(z)
& =\langle v, \overline{\phi}(U)\phi(U)f(U)v\rangle \\ 
& = \int |\phi(z)|^2f(z) \:dm_v(z).
\end{split}
\end{equation}
Since Equation \eqref{Eqn:mwtomv} is true for all $f\in C(\T)$, we can conclude that 
\[ dm_w = |\phi|^2dm_v;\]
in other words, the function $|\phi|^2$ is the Radon-Nikodym derivative of $m_w$ with respect to $m_v$.  Since $w$ was arbitrarily chosen from $H(v)$, we can conclude that for each $w\in H(v)$,
\[ m_w \ll m_v.\]

Extend the definition of $R_v$ on $\frak{A}$ in Equation \eqref{Eqn:R_v} to $L^2(m_v)$, so that 
\[ R_v(\phi) = \phi(U)v\in H(v) \subset L^2(\mu_{\frac{1}{4}})\]
for all $\phi\in L^2(m_v)$.  We will now compute the adjoint of $R_v$, and in the process, we will use the Hilbert space $\mathscr{H}(\T)$.

\begin{theorem}\label{Prop:R_v^*}Fix $v\in \mathcal{H}$ with $\|v\| = 1$.  Let 
\[R_v:L^2(m_v)\rightarrow H(v)\subset L^2(\mu_{\frac{1}{4}})\]
 be defined as $ R_v(\phi) = \phi(U)v$.  Let $w\in H(v)$.  The adjoint of $R_v$ is given by
\begin{equation}\label{Eqn:Rv*}
R_v^*(w) = \sqrt{\frac{dm_w}{dm_v}}.
\end{equation}
\end{theorem}
\noindent\textbf{Proof:  }Let $w\in H(v)$, and let $\phi\in L^2(m_v)$.  We compute $\langle R_v^*w, \phi\rangle_{L^2(m_v)}$:  
\begin{equation}
\begin{split}
\langle R_v^*w, \phi\rangle_{L^2(m_v)}
& = \langle w, R_v\phi\rangle_{L^2(\mu_{\frac{1}{4}})}\\
& = \langle (1, m_w), (\phi, m_v)\rangle_{\mathscr{H}(\T)},
\end{split}
\end{equation}
where the inner product in the last line comes from an application of the Nelson isomorphism (Subsection \ref{Subsec:NIsom}).

Since $m_w\ll m_v$, and $m_v \ll m_v$, we can use $m_v$ as the measure in the Nelson inner product in Equation \eqref{Eqn:NelsonIP}.  We calculate $R^*_v$:
\begin{equation}
\begin{split}
& \langle R_v^*w, \phi\rangle_{L^2(m_v)}\\
& = \langle (\mathbf{1}, m_w), (\phi, m_v)\rangle_{\mathscr{H}(\T)}\\
& = \int_{\T} \mathbf{1}(z)\phi(z) \sqrt{\frac{dm_w}{dm_v}}(z) \sqrt{\frac{dm_v}{dm_v}}(z) \: dm_v(z)\\
& = \int_{\T} \phi(z)  \sqrt{\frac{dm_w}{dm_v}}(z)\: dm_v(z)\\
\end{split}
\end{equation}
We conclude that $R^*_vw = \sqrt{\frac{dm_w}{dm_v}}$.
\hfill$\Box$

\begin{remark}
In general, knowing an explicit equation for the adjoint of $R_v$ is rare. 
\end{remark}

\section{Spectral properties of $U$}\label{Sec:Spectral}
In this section and the next we prove that the unitary operator $U$ on $L^2(\mu_{\frac14})$  defined from the $5$-scaled ONB  $5\Gamma$ acts ergodically, with ergodicity defined relative to $\mu_{\frac14}$ in the sense of Halmos \cite{Hal56}.  Specifically, only the constant functions are invariant under $U$.

\begin{lemma}\label{Lemma:DiracMassIFFfixed}
Suppose $v\in L^2(\mu)$ and $\|v\| = 1$.  The measure $m_v$ is a Dirac mass supported at $1$ if and only if $Uv = v$.
\end{lemma}
\textbf{Proof:  }
($\Rightarrow$)  Assume $m_v = \delta_1$, and consider the norm $\|Uv - v\|^2$.  By separating our inner product into four parts, we get
\begin{equation}
\begin{split}
\|v - Uv\|^2
& = \Bigl\langle v - Uv, v - Uv\Bigr\rangle_{L^2(\mu)}\\
& = \|Uv\|^2 + \|v\|^2 - \langle v, Uv\rangle - \overline{ \langle v, Uv\rangle}.
\end{split}
\end{equation}
Since $U$ is unitary and $\|v\|=1$, we have
\begin{equation}
\|v - Uv\|^2 = 2 - \langle v, Uv\rangle - \overline{ \langle v, Uv\rangle}.
\end{equation}
Now we take advantage of the measure $m_v$:
\begin{equation}
\begin{split}
\|v - Uv\|^2 
& = 2 - \langle v, Uv\rangle - \overline{ \langle v, Uv\rangle}\\
& = 2 - \int z \:dm_v(z) - \int \overline{z} \:dm_v(z)\\
& = 2 - \int z \:d\delta_1(z) - \int \overline{z} \:d\delta_1(z)\\
& = 2 - 1 - \overline{1} = 0.
\end{split}
\end{equation}
Therefore $\|v - Uv\|= 0$, and $v = Uv$.

\noindent ($\Leftarrow$)  Suppose $Uv = v$ for some $v$ with $\|v\|=1$.   For any $f\in L^2(m_v)$,
\begin{equation}\label{Eqn:f(U)E^u}
f(U) = \int f(z) E^U(dz).
\end{equation}
Since $Uv = v$, we find that
\begin{equation}\label{Eqn:f(1)v}
f(U)v = f(1) v.
\end{equation}
To see this, start with a polynomial:  if $f(U) = \sum c_k U^k$, then
\[ f(U)v = \Bigl(\sum c_k U^k\Bigr)v = \sum c_k U^kv = \sum c_k v = f(1)v.\]
We then use the polynomials as the starting point for approximating all other functions in $L^2(\mu)$.  In particular, let $f$ be a characteristic function $\chi_A$ for a Borel subset $A$ of the circle $\mathbb{T}$.  The right hand side above is $\chi_A(1)$.  Then by Equation \eqref{Eqn:f(1)v},  \[m_v(A) = \langle \chi_A(1)v, v \rangle = \left\{ \begin{matrix} 1 & 1 \in A \\ 0 & 1 \notin A. \end{matrix} \right.\]  In other words, $m_v$ is the Dirac mass $\delta_1$.
\hfill$\Box$
 
\begin{corollary}\label{Cor:Dirac}  Suppose $v\in L^2(\mu)$ and $\|v\| = 1$.  Then $Uv= \lambda v$ where $\lambda\in \mathbb{S}^1$ if and only if $m_v$ is a Dirac mass supported at $\lambda$.
\end{corollary}  
\textbf{Proof:  }Replace ``$1$'' in the proof above by ``$\lambda$''.\hfill$\Box$

Now, we will assume that $v$ is a non-constant function which is fixed by $U$.  It is relatively clear  that $v$ cannot actually be one of the $e_{\gamma}$'s for some $\gamma\in \Gamma\backslash\{0\}$.  If it were, then $e_{5\gamma} = e_{\gamma}$, which could be true on some finite set of points but is not true on $X_{\frac14}$.  We next consider whether $v$ can belong to a cyclic subspace generated by one of the $e_{\gamma}$'s.  As in earlier sections, we use the notation $H(v)$ to denote the $U$-cyclic subspace generated by $v \in L^2(\mu_{\frac14})$.

\begin{proposition}\label{Prop:Cyclic}  Suppose $v\in L^2(\mu)$ and $\|v\| = 1$. Choose  $\gamma\in  \Gamma\backslash\{0\}$ such that $\langle v, e_{\gamma}\rangle \neq 0$. 
 If $Uv = v$, then $v$ is not in the $U$-cyclic subspace $H(e_{\gamma})$ generated by $e_{\gamma}$.  
\end{proposition}
\textbf{Proof:  }  From Theorem \ref{Thm:Hv=Kv}, we know that the elements of $H(e_{\gamma})$ are in one-to-one isometric correspondence with the functions in $L^2(m_{e_{\gamma}})$.  Our goal will be to show that for all $f \in L^2(m_{e_{\gamma}})$, 
\[ f(U) e_{\gamma} = f(1) e_{\gamma}.\]

Suppose $v\in H(e_{\gamma})$---i.e., 
\begin{equation}\label{Eqn:v=f(U)e}
v = f(U)e_{\gamma},
\end{equation} 
where $f\in L^2(m_{e_{\gamma}})$.  Since $Uv = v$, we have $Uf(U)e_{\gamma} = f(U)e_{\gamma}$.  Therefore, using the isometric isomorphism between $H(e_{\gamma})$ and $L^2(m_{e_{\gamma}})$, we have
\begin{equation}\label{Eqn:fnonzero}
z f(z) = f(z), \textrm{ or } f(z)(z-1) = 0
\end{equation}
a.e. $m_{e_{\gamma}}$ on $\mathbb{T}$.  Therefore, $f(z)$ is nonzero at $z=1$, and $f$ is $0$ a.e. $m_{e_{\gamma}}$.  In other words, since $U$ fixes $v$, we know that $f$ is fixed by multiplication by $z$.  

\noindent\textbf{Claim 1:  }The measures $m_{e_{\gamma}}$ and $m_v$ are related in the following way:
\begin{equation}
|f(z)|^2dm_{e_{\gamma}}(z) = dm_{v}(z) = d\delta_1(z).
\end{equation}

Let $\phi\in C(\mathbb{T})$.  Then
\begin{equation}
\begin{split}
\int_{\mathbb{T}} |f(z)|^2\phi(z)dm_{e_{\gamma}}(z)
& = \langle f(U) \overline{f(U)} \phi(U) e_{\gamma}, e_{\gamma}\rangle_{L^2(\mu)}\\
& =  \langle \phi(U)  f(U) e_{\gamma},  f(U) e_{\gamma}\rangle_{L^2(\mu)}\\
& = \langle \phi(U)v, v\rangle_{L^2(\mu)} \\
& = \int \phi(z) dm_v(z).
\end{split}
\end{equation}
By Lemma \ref{Lemma:DiracMassIFFfixed}, $m_v = \delta_1$. 

\noindent\textbf{Claim 2:  } The measure $f dm_{e_{\gamma}}$ is a constant multiple of the measure $\delta_1$.
Let $\phi\in C(\mathbb{T})$.  Then
\begin{equation}
\int_{\mathbb{T}} f(z)\phi(z)dm_{e_{\gamma}}(z)
= \langle f(U)\phi(U)e_{\gamma}, e_{\gamma} \rangle_{L^2(\mu)}.
\end{equation}
Split $\mathbb{T}$ into two pieces:
\begin{equation}
\begin{split}
& \int_{\mathbb{T}} f(z)\phi(z)dm_{e_{\gamma}}(z) \\
& = \int_{\{1\}} f(z)\phi(z)dm_{e_{\gamma}}(z) +\int_{\mathbb{T}\backslash\{1\}} f(z)\phi(z)dm_{e_{\gamma}}(z)
\end{split}
\end{equation}
By Equation (\ref{Eqn:fnonzero}), we know that the integral over $\mathbb{T}\backslash\{1\}$ is $0$.  Therefore
\begin{equation}
\int_{\mathbb{T}} f(z)\phi(z)dm_{e_{\gamma}}(z) = f(1)\phi(1) m_{e_{\gamma}}(\{1\}),
\end{equation}
and $f(z)dm_{e_{\gamma}} = f(1) m_{e_{\gamma}}(\{1\})d\delta_1$.

We can combine Claims 1 and 2 to deduce that $m_{e_{\gamma}}(\{1\}) = 1$ and $|f(1)|^2 = 1$:
\begin{equation}\label{Eqn:Dirac1}
\begin{split}
d\delta_1(z)
& = |f(z)|^2 dm_{e_{\gamma}}(z)\\
& = \overline{f(z)} f(z)  dm_{e_{\gamma}}(z)\\
& = \overline{f(z)}f(1) m_{e_{\gamma}}(\{1\})d\delta_1(z)\\
& =  \overline{f(1)}f(1) m_{e_{\gamma}}(\{1\})d\delta_1(z)\\
& = |f(1)|^2m_{e_{\gamma}}(\{1\})d\delta_1(z).\\
\end{split}
\end{equation}
Since $|f|^2m_{e_{\gamma}} = \delta_1$ is a probability measure supported at $1$, we know that 
\begin{equation}\label{Eqn:f(1)}
|f(1)|^2 = 1.
\end{equation}  But
\[|f(1)|^2m_{e_{\gamma}}(\{1\}) = 1\]
by Equation (\ref{Eqn:Dirac1}), so 
\begin{equation}\label{Eqn:m_e}
m_{e_{\gamma}}(\{1\}) = 1
\end{equation} as well.

Next, we use the isometric isomorphism between $H(e_{\gamma})$ and $L^2(m_{e_{\gamma}})$ to show that $f(1)e_{\gamma} = f(U)e_{\gamma}$.   First, let $K(z) = f(z) - f(1)$ for $z\in\mathbb{T}$.
\begin{equation}
\begin{split}
\|f(U)e_{\gamma} - f(1)e_{\gamma}\|^2_{L^2(\mu))}
& = \|K(U)e_{\gamma}\|^2_{L^2(\mu)}\\
& = \|K(z)\|^2_{L^2(m_{e_{\gamma}})}\\
\end{split}
\end{equation}
Now, we look at the inner product defining $\|K(z)\|^2_{L^2(m_{e_{\gamma}})}$:
\begin{equation}
\begin{split}
&\int_{\mathbb{T}} K(z) \overline{K(z)} dm_{e_{\gamma}}\\
& = \int_{\mathbb{T}} |f(z)|^2 dm_{e_{\gamma}} + \int_{\mathbb{T}} |f(1)|^2 dm_{e_{\gamma}} - 2\textrm{Re}\overline{f(1)}\int_{\mathbb{T}} f(z) dm_{e_{\gamma}}\\
&=  \underbrace{1}_{\textrm{Claim 1}} +  \underbrace{1}_{\textrm{Eqn (\ref{Eqn:f(1)})}} - \underbrace{2 \textrm{Re} \overline{f(1)}f(1)}_{\textrm{Claim 2 and Eqn (\ref{Eqn:m_e})}} = 0.
\end{split}
\end{equation}

Finally, we show that $v$ cannot be in the cyclic subspace $H(e_{\gamma})$.  Assume $v = f(U)e_{\gamma}$.  Then
\[ v  = f(1) e_{\gamma},\]
but $U$ cannot fix any scalar multiple of an exponential function by the paragraph following Corollary \ref{Cor:Dirac}.
\hfill$\Box$

\begin{lemma}\label{Lemma:OGCyclic}
Suppose $U$ is a unitary operator on $L^2(\mu)$.  Suppose $v,w\in L^2(\mu)$.  If $v \perp H(w)$, then $H(v) \perp H(w).$
\end{lemma}
\textbf{Proof:  }Let $f\in L^2(m_w)$, and let $x = f(U)w$.  Let  $k\in \Z$.  Consider the inner product
\begin{equation}
\langle U^k v, f(U)w\rangle = \langle v, U^{-k}f(U)w\rangle.
\end{equation}
Since  $f\in L^2(m_w)$ and $m_w$ is supported on the circle $\mathbb{T}$, we also have $z^{-k}f(z)\in L^2(m_w)$.   Theorem \ref{Thm:Hv=Kv} then gives $U^{-k}f(U)w\in H(w)$.  Since $v$ is orthogonal to $H(w)$,
\begin{equation}
\langle U^k v, f(U)w\rangle = \langle v, U^{-k}f(U)w\rangle = 0.
\end{equation}
By linearity, every vector of the form $g(U)v$ where $g$ has the form
\begin{equation}\label{Eqn:Poly}
 g(z) = \sum_{n=-k}^k c_k z^k
 \end{equation}
 is orthogonal to $H(w)$.  Since the functions $g$ in Equation (\ref{Eqn:Poly}) are dense in $L^2(m_v)$,  we can conclude that $H(v) \perp H(w)$.
   \hfill$\Box$

We remark here that given $v\neq 0$, we can define a real-valued probability measure on $\mathbb{T}$ with $\widetilde{m}_v = \frac{m_v}{\|v\|^2}$, where 
\begin{equation}\label{Eqn:ProbMeas}
\widetilde{m}_v(A) = \frac{m_v}{\|v\|^2}(A) = \frac{1}{\|v\|^2} \langle E^U(A)v, v\rangle
\end{equation}
for any Borel set $A\subseteq \mathbb{T}$.
\begin{theorem}\label{Thm:U_Ergodic}  If $Uv=v$, with $\|v\|=1$, then $v = \alpha e_0$ for some $\alpha \in \mathbb{T}$, i.e. $U$ is an ergodic operator.
\end{theorem}
\textbf{Proof:  }Assume there exists $v\in L^2(\mu)\ominus \overline{\textrm{span}\{e_0\}}$ such that $Uv= v$ and $\|v\| = 1$.  Choose $\gamma\in\Gamma\backslash\{0\}$ such that $\langle v, e_{\gamma}\rangle_{L^2(\mu)} \neq 0$.  Let $Q$ be the orthogonal projection onto $H(e_{\gamma})$.

Let $v = v_1 + v_2 = Qv + v_2$, where $v_2$ is orthogonal to $v_1$.  Because $v_2$ is orthogonal to $H(e_{\gamma})$, $H(v_2) \perp H(e_{\gamma})$ by Lemma \ref{Lemma:OGCyclic}.

Let $A$ be a Borel set in $\mathbb{T}$.  Recall that $E^U$ is the projection-valued measure associated to $U$ via the Spectral Theorem.  We compute $m_v(A)$:
\begin{equation}\label{Eqn:mv}
\begin{split}
m_v(A)
 & = \langle E^U(A)v,v\rangle =  \langle E^U(A)(v_1+v_2),v_1+v_2\rangle\\
 & = \langle E^U(A)v_1, v_1\rangle + \langle E^U(A)v_1, v_2\rangle \\
 & \phantom{{=}} +  \langle E^U(A)v_2, v_1\rangle  + \langle E^U(A)v_2, v_2\rangle.
\end{split}
\end{equation}
Now, $E^U(A)v_1$ is an element of $H(v_1)$ because 
\[ E^U(A) = \int_{\mathbb{T}} \chi_A(z) dE^U(z) = \chi_{A}(U).\]
Since $v_1 \in H(e_{\gamma})$, we can write $v_1 = f(U)e_{\gamma}$ for $f$ a function in $L^2(m_{e_{\gamma}})$. Then,
\[  \langle E^U(A)v_1, v_2\rangle  = \langle \chi_A(U)f(U)e_{\gamma}, v_2 \rangle. \] The product $\chi_A \cdot f$ is again a function in $L^2(m_{e_{\gamma}})$,  so Theorem \ref{Thm:Hv=Kv} shows that
 $E^Uv_1 \in H(e_{\gamma})$.   Thus we have that the term  $\langle E^U(A)v_1, v_2\rangle = 0$ since $v_2$ is orthogonal to $H(e_{\gamma})$.  Similarly, 
 \[  \langle E^U(A)v_2, v_1\rangle =  \langle v_2, E^U(A)v_1\rangle = \langle v_2, \chi_A(U)f(U)e_{\gamma} \rangle = 0.\]
 
 This gives, for any Borel subset $A \subseteq \mathbb{T}$,  
 \begin{eqnarray*} m_v(A) &=& \langle E^U(A)v_1, v_1\rangle  + \langle E^U(A)v_2, v_2\rangle \\&=& m_{v_1}(A) + m_{v_2}(A) \\ &=& \|v_1\|^2\widetilde{m}_{v_1}(A) + \|v_2\|^2 \widetilde{m}_{v_2}(A).\end{eqnarray*}

We have shown in the above that $m_v$ is a convex combination of the probability measures $\widetilde{m}_{v_1}$ and $\widetilde{m}_{v_2}$.  The coefficients are both nonzero since the vectors $v_1$ and $v_2$ are both nonzero.  But this contradicts the fact from \ref{Lemma:DiracMassIFFfixed} that $m_v = \delta_1$ since Dirac measures are extreme points in the convex space of probability measures.
With this contradiction, we find that $v$ must be a unit vector in the span of the vector $e_0$.   Therefore the operator $U$ is ergodic.
\hfill$\Box$

\section{The mixed scales $4$ and $5$}\label{Sec:MPT}

In this section, we study the two different scales $\times 4$ and $\times 5$---scaling by $4$ and scaling by $5$.  We have devoted most of the paper to the scale $\times 5$ because $\times 5$ maps one ONB of $L^2(\mu_{\frac14})$ to another.  However, the ``natural'' scale inherent in $L^2(\mu_{\frac14})$ is $\times 4$.  For example, if $\tau_n:[0,1]\rightarrow[0,1]$ is defined by $\tau_n(x) = nx \pmod{1}$, then
\[ \mu_{\frac14}\circ \tau_4^{-1} = \mu_{\frac14}.\]
We will see that it is difficult to obtain positive results for the corresponding measure 
\[ \mu_{\frac14}\circ \tau_5^{-1}.\]

Let $U_n:L^2(\mu_{\frac14})\rightarrow L^2(\mu_{\frac14})$ defined on the ONB $E(\Gamma)$ by
\begin{equation}
U_n(e_{\gamma}) = e_{n\gamma}.
\end{equation}
As we have seen, $U_5$ is ergodic (Theorem \ref{Thm:U_Ergodic}), and $U_4$ is an isometry but not unitary.

First, we will study the spatial implementation of $U = U_5$ and $U_4$.  Then we compare ergodic theorems for the operators $U_5$ and $U_4$.  Finally, we compare the spectral measures from $U_5$ to the measure $\mu_{\frac14}$ itself.  Our results about the scaling pair $(\times 4, \times 5)$ fit into the setting of the paper \cite{JoRu95}, which explores occurrence and non-occurrence of mixed scaling in ergodic theory.

\subsection{Spatial implementation}

Recall from the Introduction (Section \ref{Subsec:Motiv}, Equations \eqref{Eqn:U^3e_1} and \eqref{Eqn:e_125}) that although it is tempting to think that $U_5^k e_{\gamma} = e_{5^k \gamma}$, this equation does not hold in general.  However, such an equation certainly holds for $U_4$.  

\begin{definition}We say that the operator $T$ is \textbf{spatially implemented} if there exists a point transformation $\tau:[0,1]\rightarrow [0,1]$ such that $Tf = f\circ\tau$.  
\end{definition}

\begin{lemma}The operator $U_5:L^2(\mu_{\frac14}) \rightarrow L^2(\mu_{\frac14})$ is not spatially implemented.
\end{lemma}
\textbf{Proof:  }Suppose there were such a transformation $\tau:[0,1]\rightarrow [0,1]$ such that $U_5f = f\circ \tau$.  Then $U_5(fg) = U_5(f)U_5(g)$, and as a specific consequence,
\begin{equation}\label{Eqn:5.1}
\begin{split}
U_5(e_{1}\cdot e_{1}) 
& = U_5(e_{1})\cdot U_5(e_{1}) = e_5\cdot e_5 = e_{10}\\
& = \widehat{\mu}(10)e_{0} + \sum_{\xi\neq 0} \widehat{\mu}(10 - \xi)e_{\xi}.
\end{split}
\end{equation}
On the other hand, $e_{1}\cdot e_{1} = e_{2}$, and 
\begin{equation}\label{Eqn:5.2}
\begin{split}
U_5(e_{2}) 
& = \sum_{\gamma\in\Gamma} \widehat{\mu}(2 - \gamma)e_{5\gamma}= \sum_{\gamma, \xi\in\Gamma}\widehat{\mu}(2 - \gamma)\widehat{\mu}(5-\xi)e_{\xi}\\
& = \sum_{\gamma\in\Gamma}\widehat{\mu}(2 - \gamma)\widehat{\mu}(5)e_{0} + \sum_{\gamma\in\Gamma, \xi\neq 0}\widehat{\mu}(2 - \gamma)\widehat{\mu}(5-\xi)e_{\xi}\\
& = 0\, e_{0} + \sum_{\gamma\in\Gamma, \xi\neq 0}\widehat{\mu}(2 - \gamma)\widehat{\mu}(5-\xi)e_{\xi}.
\end{split}
\end{equation}
By comparing the constant terms in Equations \eqref{Eqn:5.1} and \eqref{Eqn:5.2}, we see that the two expressions cannot be the same, since $\widehat{\mu}(10) \neq 0$.  Therefore $U_5$ is not spatially implemented.\hfill$\Box$

The operator $U_4$, on the other hand, is readily seen to be spatially implemented by the map $\tau_4(x) = 4x \pmod{1}$.

\subsection{Averaging}

With Theorem \ref{Thm:U_Ergodic} in hand, we can study averaging with respect to $U_5$ and $U_4$.  Suppose $T:\hs\rightarrow\hs$ is a bounded operator, and 
\[Q = \{ f\in \hs : Tf = f\}.\]  
Let $P_Q$ be the orthogonal projection onto $Q$.  The ergodic theorem of von Neumann states that 
\begin{equation}\label{Eqn:vN_erg}
\lim_{N\rightarrow\infty} \frac{1}{N+1} \sum_{k=0}^N T^k f = P_Q(f)
\end{equation}
\cite{Yo74}.  In the special case of $U_5:L^2(\mu_{\frac14})\rightarrow L^2(\mu_{\frac14})$, the subspace $Q$ is the one-dimensional space spanned by the constant function $e_0$ by Theorem \ref{Thm:U_Ergodic}.  Therefore, 
\begin{equation}\label{Eqn:vN_Erg_U}
\lim_{N\rightarrow\infty} \frac{1}{N+1} \sum_{k=0}^N U_5^k f = \langle f, e_0\rangle e_0 = \Big(\int f(x) \dd \mu_{\frac14}\Bigr)e_0.
\end{equation}
If we think of the Cesaro mean of the iterations of $U$ as a ``time average'', and think of  the integral with respect to the measure $\mu_{\frac14}$ as a  ``space average'', then we have now shown that the time average applied to functions in $L^2(\mu_{\frac14})$ equals the space average. 
  
By contrast, we note that the isometry $U_4(e_\gamma) := e_{4\gamma}$ is spatially implemented and that it is induced by $\tau_4(x) = 4x \pmod{1}$.  We also know that $\mu_{\frac14}$ is invariant under $\tau_4$.  Because $U_4$ can be realized as a shift on the underlying digit space, it is not hard to see that the only functions fixed by $U_4$ are also the constant functions:
\begin{equation}
\lim_{N\rightarrow\infty} \frac{1}{N+1} \sum_{k=0}^N U_4^k f = \langle f, e_0\rangle e_0 = \Big(\int f(x) \dd \mu_{\frac14}\Bigr)e_0.
\end{equation}
Again the time average applied to $U_4$ on the function $f$ equals the same space average of $f$. But the result for $U_5$ is much deeper than that of $U_4$.  For background references on ergodic transformations, see \cite{Yo74, Hal56}, and for references on multiplicity theory, see \cite{Hal51, Nel69}.

There is no clean relation between the two measures $\mu_{\frac14}$ and $\mu_{\frac14}\circ\tau_5^{-1}$.  
\begin{corollary}\label{Cor:Meas}
The measures $\mu_{\frac14}$ and $\mu_{\frac14}\circ\tau_5^{-1}$ are not equivalent.
\end{corollary}

\textbf{Proof:  }
Set $A = (\frac23, 1]$.  Then $\mu_{\frac14}(A) = 0$, since $A$ is not contained in the Cantor set $X_{\frac14}$, but $\mu_{\frac14}\circ\tau_5^{-1}( (A) > \frac18$, as demonstrated in Figure 1.  Therefore $\mu_{\frac14}\circ\tau_5^{-1}$ is not absolutely continuous with respect to $\mu_{\frac14}$.  Neither are the two measures concentrated on disjoint sets:  both measures assign positive values to the set $[0, \frac12]$.\hfill$\Box$

\begin{figure}
\begin{center}
\includegraphics{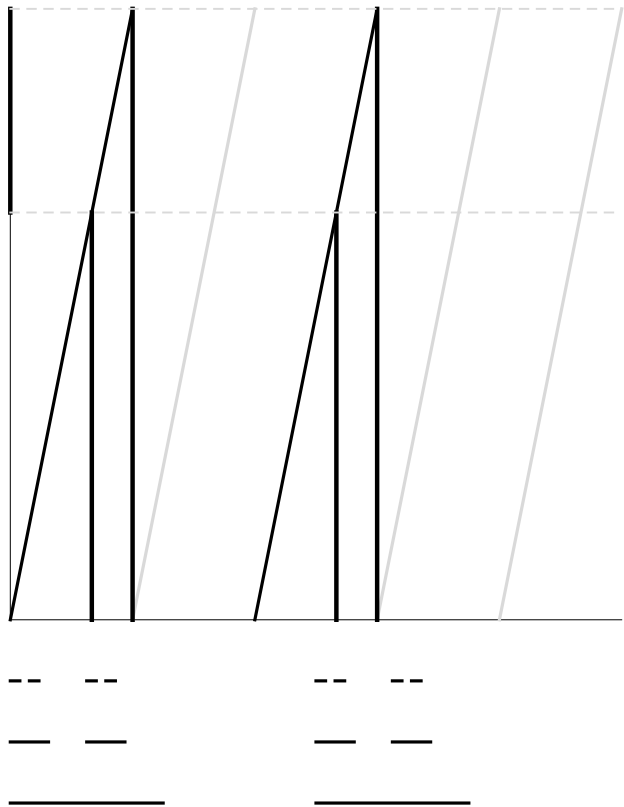}
\caption{The graph of $\tau_5$ on $[0,1]\times [0,1]$ sits above the first two approximations of the Cantor set $X_{\frac{1}{4}}$.  The set $(\frac{2}{3}, 1]$ is pulled back by two branches of $\tau_5^{-1}$.   The left-most branch of $\tau_5^{-1}$ pulls $[\frac{2}{3}, 1)$ back to a set on the horizontal axis which contains a scaled copy of $X_{\frac{1}{4}}$.}
\end{center}
\end{figure}

\renewcommand{\baselinestretch}{1.1} \large\normalsize


\end{document}